\documentclass[10pt,leqno]{amsart}
\title[Topological properties of the unitary group]
{Topological properties of the unitary group}

\author{Jes\'us  Espinoza}
\address{Departamento de matem\'aticas, Universidad de Sonora,
 Blvd Lu\'is Encinas y Rosales S/N, Colonia Centro. Edificio 3K-1. C.P. 83000. Hermosillo, Sonora, M\'exico }
 \email{goku.espinoza@gmail.com}    
\author{Bernardo Uribe}
\address{Departamento de Matem\'aticas y Estad\'istica\\ Universidad del Norte\\ Km 5 V\'ia Puerto Colombia\\ Barranquilla\\ Colombia}

\email{bjongbloed@uninorte.edu.co}

\thanks{The first author acknowledges the financial support of CONACyT and of the Universidad de Sonora. The second author acknowledges the financial support of the Alexander von Humboldt Foundation.}

\usepackage{pdfsync}
\usepackage{hyperref}
\usepackage{calc}
\usepackage{enumerate,amssymb}
\usepackage{latexsym, amscd}
\usepackage{color}
\usepackage[arrow,curve,matrix,tips,2cell]{xy}
  \SelectTips{eu}{10} \UseTips
  \UseAllTwocells
\usepackage{tikz}
\usepackage{pdfcolmk}

\DeclareMathAlphabet\EuR{U}{eur}{m}{n}
\SetMathAlphabet\EuR{bold}{U}{eur}{b}{n}

\makeindex             

\theoremstyle{plain}
\newtheorem{theorem}{Theorem}[section]
\newtheorem{lemma}[theorem]{Lemma}

\theoremstyle{definition}

{\catcode`@=11\global\let\c@equation=\c@theorem}

\newcommand{\comsquare}[8]                   
{\begin{CD}
#1 @>#2>> #3\\
@V{#4}VV @V{#5}VV\\
#6 @>#7>> #8
\end{CD}
}

\newcommand{\xycomsquare}[8]                   
{\xymatrix
{#1 \ar[r]^{#2} \ar[d]^{#4} &
#3 \ar[d]^{#5}  \\
#6\ar[r]^{#7} &
#8
}
}

\newcommand{\UH}{{U(\mathcal{H})}}

\newcommand{\BH}{{B(\mathcal{H})}}
\newcommand{\HH}{\mathcal{H}}

\newcommand{\naturals}{\mathbb{N}}

\newcommand{\higherlim}[3]{{\setbox1=\hbox{\rm lim}
        \setbox2=\hbox to \wd1{\leftarrowfill} \ht2=0pt \dp2=-1pt
        \mathop{\vtop{\baselineskip=5pt\box1\box2}}
        _{#1}}^{#2}#3}

\newcommand{\version}[1]                       
{\begin{center} last edited on #1\\
last compiled on \today\\
name of texfile: \jobname
\end{center}
}

\newcounter{commentcounter}

\begin{document}

\typeout{----------------------------  linluesau.tex  ----------------------------}


\typeout{------------------------------------ Abstract ----------------------------------------}

\begin{abstract}
We show that the strong operator topology, the weak operator topology and the compact-open topology agree 
on the space of unitary operators of a infinite dimensional separable Hilbert space. Moreover, we show that the unitary group endowed with any of these
topologies  is a Polish group.  

  \end{abstract}

\maketitle

\section*{Introduction}\label{Introduction}

The purpose of this short note is to settle some topological properties of the unitary group $\UH$ of an infinite dimensional separable Hilbert space $\HH$, 
whenever the group is endowed with the compact open topology. 

When dealing with equivariant Hilbert bundles  and its
relation with its associated unitary principal equivariant bundles (see \cite{AtiyahSegal}), one is obliged to
consider the compact-open topology on the structural group $\UH$. 

In one of the foundational papers for twisted equivariant K-theory, Atiyah and Segal claimed that the unitary
group endowed with the compact-open topology was not a topological group \cite[Page 40]{AtiyahSegal}, based on 
the fact that  the inverse map on $GL(\HH)$ is not continuous when $GL(\HH)$ is endowed with the strong operator topology.
This unfortunate claim obliged Atiyah and Segal to device a set of ingenious constructions in order to make $\UH$ into a topological
group with the desired topological properties suited for the classification of Fredholm bundles.  Nevertheless, these ingenious constructions
of Atiyah and Segal added difficulties on the quest of finding local cross sections for equivariant projective bundles, and therefore a 
clarification on the veracity of the claim was due.

The purpose of this note is to show that  the unitary group $\UH$ endowed with the compact-open topology is indeed a topological group,
moreover a Polish group, and that this topology agrees with the strong operator topology, as the weak operator topology; this is the content of Theorem \ref{thm:topologies-on-U(H)-agree} which is the main result of this note.

\section{Operator topologies on the unitary group}

Let $\HH$ be a separable Hilbert space with inner product $\langle,\rangle$ inducing the norm $|x|:=\sqrt{\langle x,x \rangle}$
for  $ x \in \HH$. Let  $L(\HH,\HH)$
 denote the vector space of linear operators from $\HH$ to $\HH$ 
 and denote by $\BH$ the vector space of bounded linear operators 
 $$\BH:= \{T \in L(\HH,\HH) \colon \exists M >0 \ \ \forall x \in \HH \ \ |Tx| \leq M|x| \}.$$
The space of bounded linear operators $\BH$ endowed with the norm
 $$|T| = \sup_{|x| \leq 1} |Tx|$$
becomes a Banach space. Denote the adjoint operator $T^*$ as the operator defined by the equality
$\langle Tx,y \rangle = \langle x,T^*y \rangle$
for all $x,y \in \HH$.

The space $\BH$ can be endowed with several operator topologies and we will list the ones that interest us in this work.
They all can be defined by specifying which are the convergent sequences. Then let $\{T_k\}$ be a sequence of linear operators on $\BH$ and denote by $T_k \to_? T$  the statement that $T_k$ converges to $T$ in the topology ?, and denote $\BH_?$ the space
of bounded operators endowed with the topology ?. In what follows the notation $x_k \to x$ denotes
that the sequence $\{ x_k \}_k$ in $\HH$ converges to $x$.
\begin{itemize}
\item{\bf{Norm Topology (Uniform convergence):} }\\ $T_k \to_n T$ if $\sup_{ |x|  \leq 1 }\{ |T_kx -Tx|   \} \to 0$.
\item{\bf{Strong Operator Topology (Pointwise convergence):}}\\ $T_k \to_s T$ if for all $x \in \HH$, $T_kx \to Tx$.
\item{\bf{Weak Operator Topology:}}\\ $T_k \to_w T$ if for all $x,y \in \HH$, $\langle T_kx,y \rangle  \to \langle Tx,y \rangle$
\item{\bf{Compact Open Topology (Uniform convergence on compact sets):}}\\ $T_k \to_{co} T$ if for all compact subsets $C \subset \HH$, the restricted sequence
$\{T_k|_C\}_k$ converges uniformly to $T|_C$.
\item{\bf{Strong* Operator Topology:}}\\ $T_k \to_{s^*} T$ if  both $T_k \to_{s} T$ and  $T_k^* \to_{s} T^*$.
\item{\bf{Weak* Operator Topology:}}\\ $T_k \to_{w^*} T$ if both $T_k \to_{w} T$ and  $T_k^* \to_{w} T^*$.
\item{\bf{Compact Open* Topology:}}\\ $T_k \to_{co^*} T$ if both $T_k \to_{co} T$ and  $T_k^* \to_{co} T^*$.
\end{itemize}

The identity map on $\BH$ induces the following commutative diagram of continuous maps
\begin{equation}
\xymatrix{
\BH_n \ar[r] & \BH_{co^*} \ar[d] \ar[r] & \BH_{s^*} \ar[d] \ar[r] & \BH_{w^*} \ar[d]  \\
 & \BH_{co} \ar[r] & \BH_{s} \ar[r] & \BH_{w} 
} \label{diagram_B(H)}
\end{equation}
with the property that none of the maps is a homeomorphism.

Let $\UH$ denote the group of unitary operators on $\HH$, i.e.
$$\UH=\{ U \in \BH \colon U U^*=U^*U= \rm{Id}_\HH \},$$
and note that $U^{-1}=U^*$ and that 
$\langle x,y \rangle = \langle Ux,Uy \rangle$ for all $x,y \in \HH$. Since $\UH \subset \BH$, the group $\UH$
can be endowed with any of the operator topologies previously defined. The group $\UH_n$ endowed with the 
norm topology is a topological group and it is the prototypical example of what is known as a Banach Lie group, see 
\cite{Neeb}. The fact that $\UH$ endowed with any of the other topologies defined above is also a topological group
is the main result of this section 

\begin{theorem} \label{thm:topologies-on-U(H)-agree}
The operator topologies: compact open, strong, weak and their * counterparts, all agree on the group $\UH$, i.e.
$$\UH_{co^*} = \UH_{s^*} = \UH_{w^*} = \UH_{co} = \UH_{s} = \UH_{w}.$$
Moreover, the group $\UH$ endowed with any of these topologies is a Polish group, i.e. a completely metrizable 
topological group. 
\end{theorem}

The proof of the theorem is built out from various Lemmas:

\begin{lemma} \label{lem:adjoint-continuous}
The map $\UH_w \to \UH_w$, $T \mapsto T^*$ is continuous. In particular $\UH_{w^*}=\UH_w$.
\end{lemma}
\begin{proof}
Consider $T_k \to_wT$ and $x,y \in \HH$. Then
$$\langle x, T_k^*y \rangle = \langle T_kx,y \rangle \to \langle Tx,y \rangle = \langle x, T^*y \rangle,$$
and therefore $T^*_k \to_w T^*$. This proves that $T \mapsto T^*$ is continuous. Hence we conclude
that the weak and the weak* topologies agree on $\UH$. 
\end{proof}

\begin{lemma} \label{lem:composition-continuous}
The composition of operators $\UH_s \times \UH_s \to \UH_s$, $(T,S)\mapsto TS$,  is continuous in the strong operator topology.
\end{lemma}
\begin{proof}
Consider a convergent sequence $(T_k,S_k) \to_s (T,S)$ in $\UH_s\times \UH_s$. For $x \in \HH$ we have 
\begin{align*}
|T_kS_kx - TSx| & = |T_kS_kx - T_kSx + T_kSx -TSx|\\
& \leq  |T_kS_kx - T_kSx| + |T_kSx -TSx|\\
&  = |T_k(S_kx-Sx) | + |(T_k-T)Sx|\\
& =  |S_kx-Sx| + |(T_k-T)Sx|,
\end{align*}
and since $S_kx \to Sx$ and $T_k(Sx) \to T(Sx)$, then we have that $(T_kS_k)x \to (TS)x$.
Therefore the composition of operators is continuous in the strong operator topology.
\end{proof}
\begin{lemma}\label{lem:weak=strong}
The weak and the strong operator topologies agree on $\UH$. Therefore we have that
$\UH_s$ is a topological group and moreover that $\UH_{s^*}=\UH_s=\UH_w$.
\end{lemma}
\begin{proof}
Consider a convergent sequence $T_k \to_w T$ in $\UH_w$. To prove that $T_k \to_s T$ it is enough to show the convergence
$T_kx \to Tx$ for $x$ a unit vector in $\HH$. We have then that $|x|=1$ and therefore $|T_kx|=1=|Tx|$. Now we compute
\begin{align*}
|T_kx-Tx|^2 = |T_kx|^2 -2{\rm{Re}}\langle T_kx,Tx \rangle + |Tx|^2 
= 2 -2{\rm{Re}}\langle T_kx,Tx \rangle
\end{align*}
and since $T_k\to_wT$, we have that $\langle T_kx,Tx \rangle \to \langle Tx,Tx \rangle=1$. Therefore 
$|T_kx-Tx|^2 \to 0$ and hence we conclude that $T_kx \to Tx$. This shows that the identity map $\UH_w \to \UH_s$ is continuous, and 
therefore $\UH_w=\UH_s$.

Now, by Lemma \ref{lem:adjoint-continuous} we know that the inverse map $T \mapsto T^*=T^{-1}$ is continuous, and by Lemma
\ref{lem:composition-continuous} we know that the composition is continuous. Then we have that $\UH$ with the strong (or weak) operator
topology is a topological group; the fact that the strong* topology agrees with the strong topology follows from the continuity of the inverse map. 

\end{proof}

The proofs of the previous lemmas follow the proofs that appear in  \cite[Cor. 9.4]{Hilgert-Neeb}. 

\begin{lemma}
The topological group $\UH_s$ is metrizable.
\end{lemma}
\begin{proof}
Consider an orthonormal basis $\{e_j\}_{j \in \naturals}$ of $\HH$. Consider the map
$$\Psi : \UH_s \to \HH^\naturals, \ \ T \mapsto (Te_j)_{j \in \naturals}$$
where $\HH^\naturals$ is endowed with the product topology. For a convergent sequence
$T_k \to_s T$ we have that $T_ke_j \to Te_j$ for all $j \in \naturals$, this implies that
the maps $\UH_s \to \HH, \ T\mapsto Te_j$ are all continuous and therefore by the universal properties of the
product topology we obtain that the map $\Psi$ is continuous. Since the basis $\{e_j\}_{j \in \naturals}$
generates a dense subset of $\HH$, we have that any two operators which agree on the basis $\{e_j\}_{j \in \naturals}$
must be equal; hence we have that the map $\Psi$ is injective.

Now let us show that the map $\Psi$ is an embedding. For this purpose let us take a convergent
sequence  $\Psi(T_k) \to \Psi(T)$ in the image of $\Psi$. Since $T_ke_j \to Te_j$ for all basis vectors,
we  have that the sequence $\{T_k\}_k$ converges pointwise to $T$ in the dense subset of $\HH$ generated
by the basis $\{e_j\}_{j \in \naturals}$, and since the operators are unitary we can conclude that $T_k \to_s T$. Let
us see in more detail this argument: take $x \in \HH$ and let $\{x_j\}_{j \in \naturals}$ with each $x_j$ belonging
to the dense subset generated by $\{e_j\}_{j \in \naturals}$ and such that $x_j \to x$. We compute
\begin{align*}
|T_kx-Tx| & \leq |T_kx - T_kx_j| + |T_kx_j - Tx_j| + |Tx_j -Tx| \\
&= |T_k(x-x_j)| + |T_kx_j - Tx_j| + |T(x_j -x)| \\
&= 2|x-x_j| + |T_kx_j - Tx_j|, 
\end{align*}
and since $T_kx_j \to Tx_j$, we have that $T_kx \to Tx$ and therefore $T_k \to_s T$. This implies that the map $\Psi$
induces a homeomorphism with its image, and hence it is an embedding.

The Hilbert space $\HH$ is a metric space, and the product $\HH^\naturals$ of countable copies of $\HH$ can be endowed
with a metric.  Since $\Psi: \UH_s \to \HH^\naturals$ is an embedding, then $\UH_s$ inherits the induced metric and hence it is 
metrizable. 
\end{proof}
The previous argument follows the proof that appears in \cite[Prop II.1]{Neeb}.
In \cite[Prop II.1]{Neeb} it is also shown that the the map $\Psi$ also provides an embedding
${\rm{Iso}}(\HH) \to \HH^\naturals$ where ${\rm{Iso}}(\HH):= \{T \in \BH \colon T^*T=\rm{Id}_\HH \}$ is the monoid of 
all isometries of $\HH$, which is moreover complete with respect to the induced metric. It is furthermore shown
that $\UH = \{T \in {\rm{Iso}}(\HH) \colon T^* \in {\rm{Iso}}(\HH) \}$ is a $G_\delta$ set in ${\rm{Iso}}(\HH)$, and since
$G_\delta$ sets on a complete metrizable space are complete in the induced metric \cite[Thm 1, p. 93]{Schwartz} this implies that
$\UH_s$ is a complete metrizable space.
\begin{lemma}[{\cite[Prop II.1]{Neeb}}] \label{lem:U(H)-polish} The unitary group $\UH_s$ is a Polish group.
\end{lemma}

So far we have shown that $\UH_s$ is a Polish group and  we have that $\UH_s=\UH_{s^*}=\UH_w=\UH_{w^*}$. Note in particular that
$\UH_s$ is compactly generated since it is a metrizable space. 

 Let us now see the relation with the compact open topology on $\UH$.

\begin{lemma} \label{lem:compact-open=strong}
The compact open topology and the strong operator topology agree on $\UH$, i.e. $\UH_{co}=\UH_s$.
\end{lemma}
\begin{proof}

Let us start by recalling the retraction functor $k$ defined in \cite[Def. 3.1]{Steenrod}. For $X$ a Hausdorff topological space the associated
compactly generated space $k(X)$ is the set $X$ whose topology is defined as follows: a set in $k(X)$ is closed, if its intersection
with every compact set in $X$ is itself closed. By \cite[Thm 3.2]{Steenrod} we know that the identity map $k(X) \to X$ is a continuous map,
that $k(X)$ is compactly generated, that $k(X)$ and $X$ have the same compact sets and that $X=k(X)$ whenever $X$ is compactly generated.

Let $C$ be a compact subset in $\UH_s$. Take a sequence of operators $\{T_n\}_n$ in $C$, since $C$ is compact in $\UH_s$ there exists a convergent subsequence $\{T_{n_k}\}_k$, and since $\UH_s$ is completely metrizable and hence $C$ is furthermore closed, this convergent subsequence
$T_{n_k} \to_s T$ converges to an operator $T \in C$. By the Bannach-Steinhaus Theorem (see \cite[Cor. of Thm 33.1, p. 348]{Treves}), we have that the 
sequence $\{T_{n_k}\}_k$ converges uniformly on every compact set of $\HH$, and therefore we have that $T_{n_k} \to_{co} T$. Thus
the space $C$ is also compact in $\UH_{co}$, and by the same argument as before, the induced topology of $C$ in $\UH_{co}$ agrees with the induced topology of $C$ in $\UH_s$.

We conclude that the spaces $\UH_{co}$ and $\UH_s$ have the same compact sets with the same induced topologies. This implies that
the retraction functor $k$ applied on the map $\UH_{co} \to \UH_s$ induces a homeomorphism
$$k(\UH_{co}) \stackrel{\cong}{\to} k(\UH_s).$$
Then we have the commutative diagram
$$\xymatrix{
k(\UH_{co}) \ar[r] \ar[d]^\cong  &\UH_{co} \ar[d]\\
k(\UH_{s}) \ar[r]^\cong & \UH_s
}$$ 
which implies that $\UH_{co}=\UH_s$.
\end{proof}

\begin{proof}[proof of Theorem \ref{thm:topologies-on-U(H)-agree}]
By Lemma \ref{lem:weak=strong} we know that $\UH_{s^*}=\UH_s=\UH_w$, so the vertical arrow in the middle and the bottom right horizontal arrow of  diagram \eqref{diagram} are homeomorphisms; by Lemma \ref{lem:compact-open=strong} we know that $\UH_{co}=\UH_s$, and then the bottom left horizontal arrow is also a homeomorphism. The proof of Lemma \ref{lem:compact-open=strong}
can also be used to show that $\UH_{co^*}=\UH_{s^*}$. Finally, the right vertical arrow is also a homeomorphism because of Lemma \ref{lem:adjoint-continuous}. Hence diagram \eqref{diagram_B(H)} restricted to $\UH$ becomes

\begin{equation} \xymatrix{
\UH_n \ar[r] & \UH_{co^*} \ar[d] \ar[r]^{\cong} & \UH_{s^*} \ar[d]^{\cong} \ar[r] & \UH_{w^*} \ar[d]^{\cong}  \\
 & \UH_{co} \ar[r]^{\cong} & \UH_{s} \ar[r]^{\cong} & \UH_{w} } \label{diagram}
\end{equation}
Therefore, besides $\UH_n$, every topological group in diagram \eqref{diagram} can be connected using a homeomorphism with the Polish group $\UH_s$. This finishes the proof of Theorem \ref{thm:topologies-on-U(H)-agree}.
\end{proof}

The norm topology on $\UH$ has strictly more open sets than the strong operator topology as can be easily checked with
the following sequence of operators.

 Consider $\HH:= l^2({\mathbb{N}})$ and take the sequence of operators
 $T_k:\HH \to \HH$ such that for all  $x=(\xi_n)\in\HH$ the $n$-th coordinate of $T_kx$ is defined by
$$ (T_k x)_n = \begin{cases}
            \xi_n & n \neq k \\
            \sqrt{-1}\xi_n & n = k.
           \end{cases}
$$ It follows that  $T_k^{*}=T_k^{-1}$ and therefore $\{ T_k \}_k \subset \UH$.
In the strong operator topology the sequence $\{T_k\}_k$ converges to the identity operator $\rm{Id}: \HH \to \HH$, since we have that
$$|T_k x - {\rm Id}x | =|i\xi_k - \xi_k| = |\xi_k||i - 1| \to 0 $$ for all $x=(\xi_n)\in\HH$.
On the other hand, if $ x_k = (\xi_n^{(k)}) \in \HH$ is defined by 
$$ \xi_n^{(k)} = \begin{cases}
            0 & n \neq k \\
            i & n = k,
           \end{cases}
$$ then $|T_kx_k - {\rm Id}x_k| = |-1-i| = \sqrt{2}$ and it follows that $\sup_{|x|\leq 1} \{ |T_kx - {\rm Id}x| \} \not\to 0$ whenever $k\to \infty$. Hence we have that $T_k \to_{s} {\rm Id}$ but $T_k \not \to_n {\rm Id}$.

Therefore we conclude that the operator
topologies defined at the beginning of the chapter reduced to only two once restricted to the unitary group. The unitary
group with the norm topology, making $\UH_n$ into a Banach Lie group, and the strong operator topology, making $\UH_s$
into a Polish group.

\typeout{-------------------------------------- References  ---------------------------------------}
\bibliographystyle{plain}
\bibliography{Topologies-on-U(H)}

\def\cprime{$'$}
\begin{thebibliography}{1}

\bibitem{AtiyahSegal}
Michael Atiyah and Graeme Segal.
\newblock Twisted {$K$}-theory.
\newblock {\em Ukr. Mat. Visn.}, 1(3):287--330, 2004.

\bibitem{Hilgert-Neeb}
Joachim Hilgert and Karl-Hermann Neeb.
\newblock {\em Lie semigroups and their applications}, volume 1552 of {\em
  Lecture Notes in Mathematics}.
\newblock Springer-Verlag, Berlin, 1993.

\bibitem{Neeb}
Karl-Hermann Neeb.
\newblock On a theorem of {S}. {B}anach.
\newblock {\em J. Lie Theory}, 7(2):293--300, 1997.

\bibitem{Schwartz}
Laurent Schwartz.
\newblock {\em Radon measures on arbitrary topological spaces and cylindrical
  measures}.
\newblock Published for the Tata Institute of Fundamental Research, Bombay by
  Oxford University Press, London, 1973.
\newblock Tata Institute of Fundamental Research Studies in Mathematics, No. 6.

\bibitem{Steenrod}
N.~E. Steenrod.
\newblock A convenient category of topological spaces.
\newblock {\em Michigan Math. J.}, 14:133--152, 1967.

\bibitem{Treves}
Fran{\c{c}}ois Tr{\`e}ves.
\newblock {\em Topological vector spaces, distributions and kernels}.
\newblock Academic Press, New York, 1967.

\end{thebibliography}
\end{document}